\documentclass[12pt, reqno]{amsart}

\usepackage{amsmath, amssymb, amscd, amsthm, comment, amsxtra}
\usepackage{graphicx}
\usepackage{subfigure}

\usepackage[pdftex, linktocpage=true]{hyperref}
\usepackage{euscript}

\usepackage[format=plain,labelfont=bf,up,width=.99\textwidth]{caption}

\newcommand{\cRG}{{\EuScript R}}

\theoremstyle{plain}
\newtheorem{thm}{Theorem}[section]

\newtheorem{lem}[thm]{Lemma}
\newtheorem{prop}[thm]{Proposition}
\newtheorem{remark}[thm]{Remark}
\theoremstyle{definition}
\newtheorem{defn}[thm]{Definition}

\theoremstyle{plain}
\newtheorem{thmx}{\bf Theorem}

\newcommand{\lemref}[1]{Lemma~\ref{#1}}

\newcommand{\propref}[1]{Proposition~\ref{#1}}

\newcommand{\C}{\mathbb{C}}

\newcommand{\D}{\mathbb{D}}

\newcommand{\h}{\hat}

\newcommand{\diam}{\operatorname{diam}}
\newcommand{\dist}{\operatorname{dist}}

\renewcommand{\mod}{\operatorname{mod}}
\newcommand{\tl}{\tilde}

\newcommand{\eps}{\epsilon}
\newcommand{\cQ}{{\mathcal Q}}

\newcommand{\cA}{{\mathcal A}}
\newcommand{\cU}{{\mathcal U}}

\newcommand{\cV}{{\mathcal V}}

\newcommand{\cI}{{\mathcal I}}

\newcommand{\cP}{{\mathcal P}}
\newcommand{\cC}{{\mathcal C}}
\newcommand{\cH}{{\mathcal H}}
\newcommand{\cR}{{\mathcal R}}

\newcommand{\cB}{{\mathcal B}}

\newcommand{\cS}{{\mathcal S}}

\newcommand{\CC}{{\Bbb C}}
\newcommand{\RR}{{\Bbb R}}
\newcommand{\TT}{{\Bbb T}}
\newcommand{\ZZ}{{\Bbb Z}}
\newcommand{\NN}{{\Bbb N}}
\newcommand{\DD}{{\Bbb D}}

\newcommand{\QQ}{{\Bbb Q}}

\newcommand{\bC}{{\mathbf C}}

\newcommand{\cren}{\cR_{\text cyl}}

\title[Renormalization and Siegel disks for complex H\'enon maps]{Renormalization and Siegel disks for\\ complex H\'enon maps}
\author{Denis Gaidashev}
\address{Uppsala University, Uppsala, Sweden}
\email{gaidash@math.uu.se}

\author{Remus Radu}
\address{Stony Brook University, Stony Brook, United States}
\email{rradu@math.stonybrook.edu}

\author{Michael Yampolsky}
\address{University of Toronto, Toronto, Canada}
\email{yampol@math.utoronto.ca}


\begin{document}
\begin{abstract}
We use hyperbolicity of golden-mean renormalization of dissipative H{\'e}non-like maps to prove that the boundaries of Siegel disks of sufficiently dissipative quadratic complex H{\'e}non maps with golden-mean rotation number are topological circles. 

Conditionally on an appropriate renormalization hyperbolicity 
property, we  derive the same result for Siegel disks of  H{\'e}non maps with all eventually periodic rotation numbers.
\end{abstract}
\maketitle

\medskip
\noindent
\section{Introduction}\label{sec:intro}
Consider the complex quadratic H{\'e}non map written as
\[
H_{c,a}(x,y)=(x^{2}+c + ay,ax) \text{ for } a\neq 0.
\]
The maps $H_{c,a}$ and $H_{c,-a}$ are conjugate by the change of coordinates $(x,y)\mapsto (x,-y)$; and the pair of parameters $(c,a^2)$ determines the H{\'e}non map uniquely up to a biholomorphic conjugacy. 
In this parametrization the Jacobian is $-a^{2}$. 
Let $K^{\pm}$ be the sets of points that do not escape to infinity under forward, respectively backward iterations of the H{\'e}non map. Their topological boundaries are $J^{\pm}=\partial K^{\pm}$. Let $K=K^{+}\cap K^{-}$ and $J=J^{-}\cap J^{+}$. The sets $J^{\pm}$, $K^{\pm}$ are unbounded, connected sets in $\C^{2}$ (see \cite{BS1}). The sets $J$ and $K$ are compact (see \cite{HOV1}).  In analogy to one-dimensional dynamics, the set $J$ is called the Julia set of the H{\'e}non map. 

In this paper we will always assume that the H{\'e}non map is dissipative, $|a|<1$. Note that for $a=0$, the map $H_{c,a}$ degenerates to 
$$(x,y)\mapsto (f_c(x),0),$$ where $f_c(x)=x^2+c$ is a one-dimensional quadratic polynomial. Thus for a fixed small value of $a_0$, the one parameter family $H_{c,a_0}$ is a small perturbation of the quadratic family. 

 Note that  a  H{\'e}non map $H_{c,a}$ is determined by the multipliers $\lambda$ and $\mu$ at a fixed point uniquely up to changing the sign of $a$. In particular,
$$\lambda\mu=-a^2,$$
 the parameter $c$ is a function of $a^2$ and $\lambda$:
$$c=(1-a^{2})\left(\frac{\lambda}{2}-\frac{a^{2}}{2\lambda}\right)-\left(\frac{\lambda}{2}-\frac{a^{2}}{2\lambda}\right)^{2}.$$
Hence, we sometimes  write $H_{\lambda,\mu}$ instead of $H_{c,a}$, when convenient.
 When $\mu=0$, the H{\'e}non map degenerates to 
\begin{equation}
\label{Plambda} H_{\lambda,0}(x,y) = (P_{\lambda}(x),0)\text{, where }P_{\lambda}(x)=x^{2}+\lambda/2-\lambda^{2}/4.
\end{equation}

We say that a dissipative H{\'e}non map $H_{c,a}$ has a {\it semi-Siegel fixed point}  (or simply that $H_{c,a}$ is semi-Siegel) if the  eigenvalues of the linear part of $H_{c,a}$ at that fixed point are $\lambda=e^{2\pi i \theta}$, with $\theta\in(0,1)\setminus \QQ$ and $\mu$, with $|\mu|<1$, and $H_{c,a}$ is locally biholomorphically conjugate to the linear map 
$$L(x,y)=(\lambda x,\mu y).$$
The classic theorem of Siegel \cite{Sie} states, in particular, that $H_{\lambda,\mu}$ is semi-Siegel whenever $\theta$ is Diophantine, that is $q_{n+1}<cq_n^d$, where $p_n/q_n$ are the continued fraction convergents of $\theta$. 
The existence of a linearization is a local result, however, 
in this case there exists a linearizing biholomorphism $\phi:\D\times \C\rightarrow \C^{2}$ sending $(0,0)$ to the semi-Siegel fixed point,
$$H_{\lambda,\mu}\circ \phi=\phi\circ L,$$
such that the image $ \phi(\D\times\C)$ is {\it maximal} (see \cite{MNTU}).
 We call  $ \phi(\D\times\C)$ the {\it Siegel cylinder}; it is a connected component of the interior of $K^+$ and its boundary coincides with $J^+$ (see \cite{BS2}). We let 
$$\Delta=\phi(\D\times\{0\}),$$
and by analogy with the one-dimensional case call it the {\it Siegel disk} of the H{\'e}non map. Clearly, the Siegel cylinder is 
equal to the stable manifold $W^s(\Delta)$, and 
  $\Delta\subset K$ (which is always bounded). 
Moreover,  $\partial \Delta\subset J$, the Julia set of the H{\'e}non map.

\begin{remark}
Let $\textbf{q}$ be the semi-Siegel fixed point of the H{\'e}non map. Then $\Delta\subset W^{c}(\textbf{q})$, the center manifold of $\textbf{q}$ (see e.g. \cite{S} for a definition of $W^c$). The center manifold is not unique in general, but all center manifolds $W^{c}(\textbf{q})$ coincide on the Siegel disk. This phenomenon is nicely
illustrated in \cite{O}, Figure 5.
\end{remark}


The main result of this paper is the following theorem:

\begin{thmx}\label{thm:A} There exists $\delta>0$ such that the following holds.
Let $\theta_*=(\sqrt{5}-1)/2$ be the inverse golden mean, $\lambda_*=e^{2\pi i\theta_*}$, and let $|\mu|<\delta$. Then the boundary of the Siegel disk of $H_{\lambda_*,\mu}$ is a homeomorphic image of the circle.
\end{thmx}

\noindent
By Carath{\'e}odory Theorem, the linearizing map 
\begin{equation}
\label{eq:lin}
\phi:\D\times \{0\}\rightarrow \Delta
\end{equation}
extends continuously and injectively to the boundary. However, we note:

\begin{thmx}\label{thm:B}
The conjugacy $$\phi:S^1\times \{0\}\rightarrow \partial \Delta$$ is not $C^1$-smooth.
\end{thmx}

It is worthwhile mentioning that if we assume that $\lambda=e^{2\pi i \theta}$, $\mu=e^{2\pi i \theta'}$ and the pair $(\theta,\theta')$ satisfies the two-dimensional Brjuno condition \cite{Brj}, then the conservative H\'enon map $H_{\lambda,\mu}$ has a bounded maximal domain of linearization, called a {\it Siegel ball}.  Herman \cite{He} asked {\it whether the boundary of the Siegel ball is a topological or perhaps a $C^{\infty}$ submanifold of $\C^{2}$}. We answer similar questions, in the dissipative setting, as outlined above. 

The proofs of Theorems \ref{thm:A} and \ref{thm:B}
are based on a renormalization theory for two-dimensional dissipative H{\'e}non-like maps, developed by the first and third authors in \cite{GaYa2}. 
A H\'enon-like map (see \cite{dCLM}) $H:\C^{2}\rightarrow\C^{2}$ can be defined as $H(x,y) = (f(x)+\epsilon(x,y), ax)$, for some small $\epsilon$. In this normalization, it has Jacobian $-a \partial{\epsilon}/\partial{y}$ and it reduces to the standard H\'enon map when $f(x)=x^2+c$ and $\epsilon(x,y) = ay$. In general, the Jacobian of a H\'enon-like map is not constant. Following \cite{LRT}, we say that a H\'enon-like map $H$ has a semi-Siegel fixed point if there exists a local holomorphic change of variables $\phi$ such that $\tilde{H} = \phi\circ H \circ \phi^{-1}$ is a skew product of the form $\tilde{H}(x,y)= (\lambda x, \mu(x) y)$, for some holomorphic function $\mu(x) = \mu + O(x)$, where $\lambda=e^{2\pi i \theta}$, with $\theta\in(0,1)\setminus \QQ$, and $|\mu|<1$. This condition is equivalent to the existence of a one-dimensional Siegel disk  $\Delta = \phi(\D\times\{0\})$. 

Below, we will be using several different renormalization operators. The first of them is 
the renormalization of pairs of two-dimensional dissipative maps  introduced in \cite{GaYa2}. We will recall its definition in \S~\ref{sec:RenACM}.

 In one complex dimension, it corresponds to the renormalization of {\it commuting pairs} $\cR$ (cf. \cite{Stir}). In particular, suppose that
an analytic map $f$  has a fixed Siegel disk $\Delta_f$, with a rotation number $\theta\in(0,1)$. Suppose furthermore,
that $\partial\Delta_f$ is a Jordan curve, and that there is a neighbourhood of $\overline{\Delta_f}$ in which the only critical point of $f$ is a simple critical point $c_f\in\partial\Delta_f$.
The example to keep in mind is a polynomial
$P_\lambda$, defined in (\ref{Plambda}) with $\lambda=e^{2\pi i\theta}$, such that the rotation number $\theta$ is of bounded type \cite{Pet,Yam-bounds}.

Let a number $\theta\in(0,1)$ and denote $\theta_0=\theta,\;\theta_1,\;\theta_2,\ldots$ its orbit under the  Gauss map
$$G(x)=\left\{ \frac{1}{x}\right\};$$
which is finite if
and only if $\theta\in\mathbb Q$.
We denote $r_k$ the integer part $[ 1/\theta_k ]$.
Then the numbers $r_k$ form a finite or infinite continued fraction expansion of $\theta$, which we abbreviate as
$\theta=[r_0,r_1,\ldots].$
As usual, the $n$-th continued fraction convergent of $\theta$ will be denoted by $p_n/q_n\equiv[r_0,\ldots r_{n-1}]$.

The $n$-th {\it pre-renormalization} 
$p\cR^n f$ is the restriction of the pair of iterates $(f^{q_{n+1}},f^{q_n})$ to appropriate  neighborhoods of the critical point $c_f$.
Let $\kappa(z)=\bar z$ denote the complex conjugation, and set
$$\upsilon_n(z)\equiv  (f^{q_n}(c_f)-c_f)\cdot \kappa^{\circ n}(z)+c_f;$$
this is a linear map if $n$ is even, and an anti-linear map if $n$ is odd.
The $n$-th renormalization is obtained by rescaling $p\cR^n f$ by $\upsilon_n$:
$$\cR^nf=(\upsilon_n^{-1}\circ f^{q_{n+1}}\circ\upsilon_n,\upsilon_n^{-1}\circ f^{q_{n}}\circ\upsilon_n).$$

\begin{figure}[htb]
\begin{center}
\subfigure{\includegraphics[scale=0.45]{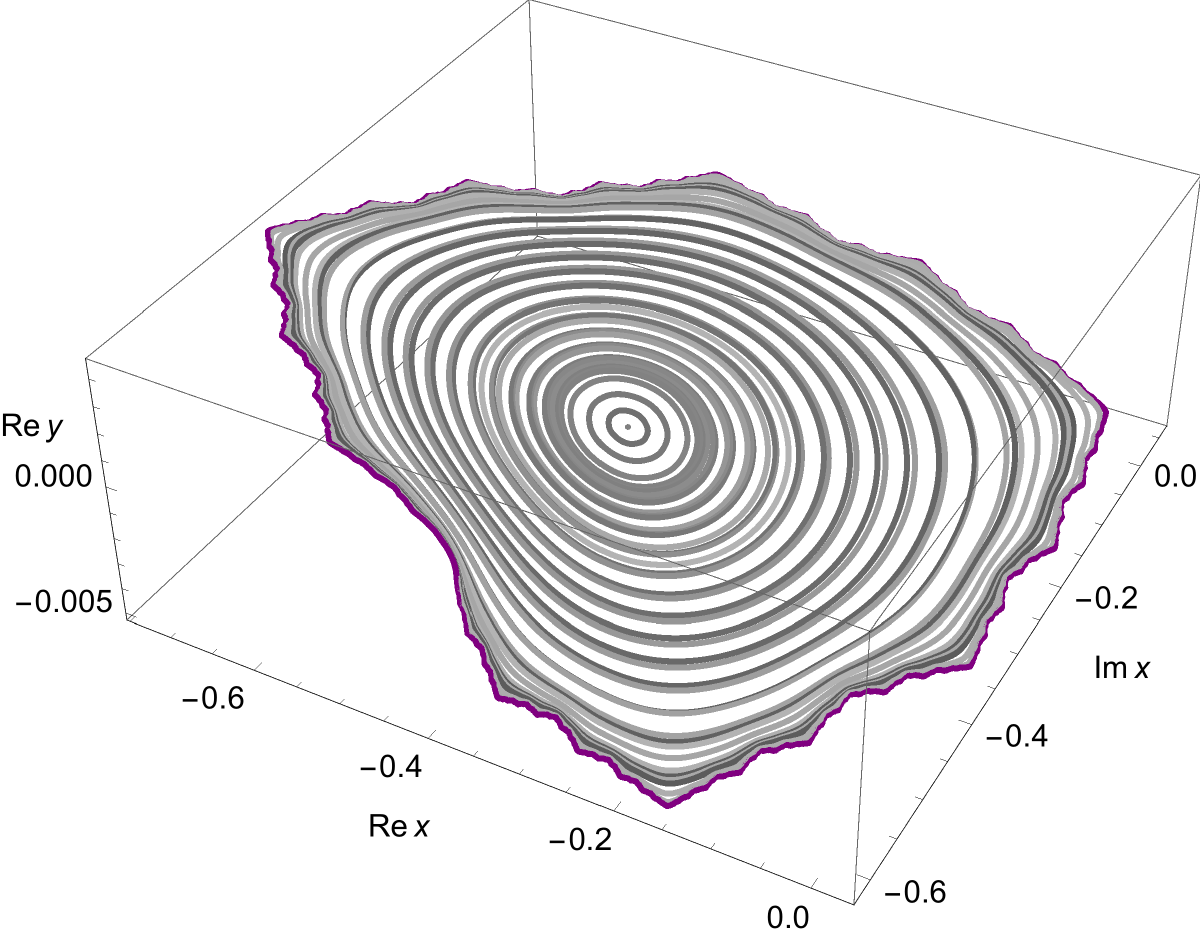}}
\subfigure{\includegraphics[scale=0.45]{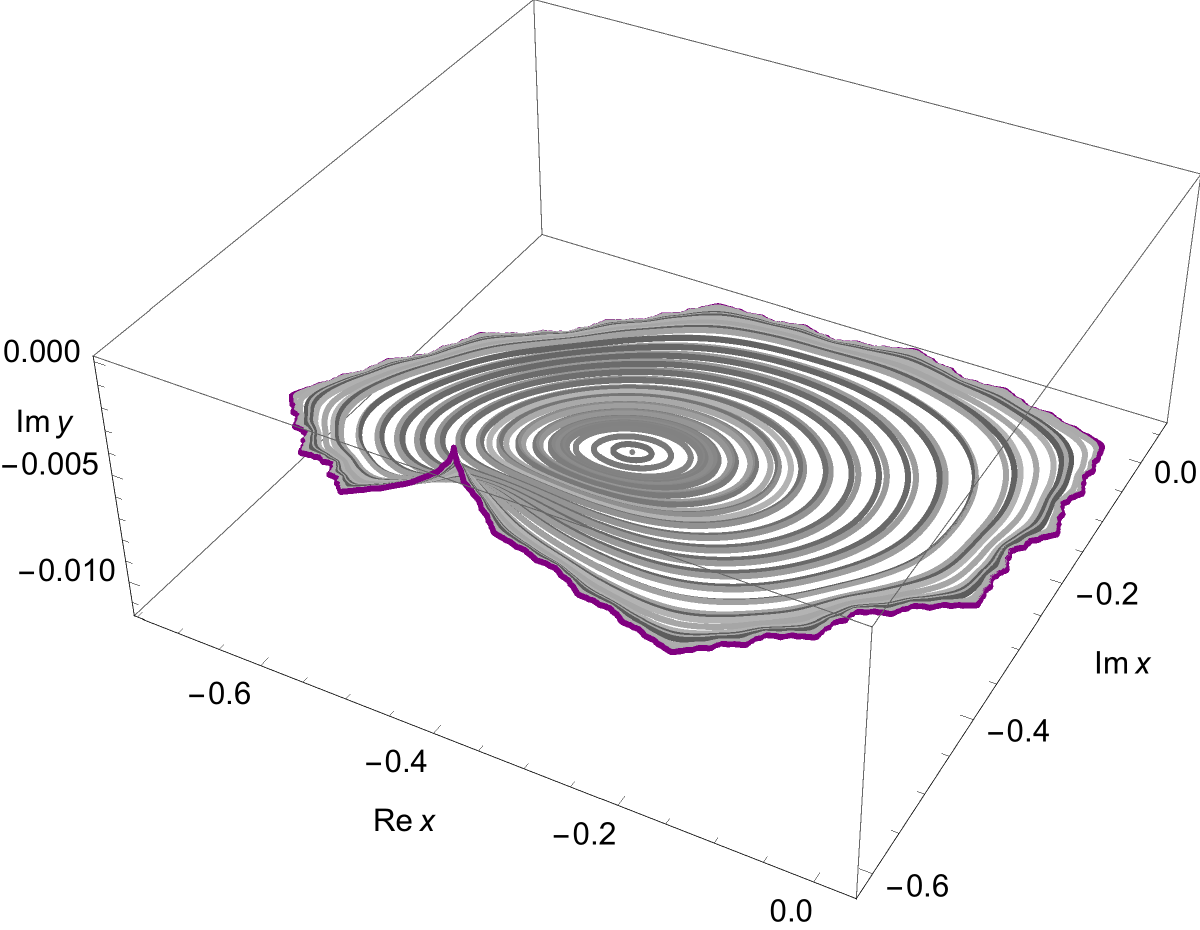}}
\end{center}
\caption{A three dimensional plot of the Siegel disk and its boundary for a H{\'e}non map with a semi-Siegel fixed point with the golden mean rotation number. The parameter $a=0.01+0.01i$. The three axes are as follows: {\sc Top:} $\mathrm{Re}(x)$, $\mathrm{Im}(x)$ and $\mathrm{Re}(y)$; {\sc Bottom:} $\mathrm{Re}(x)$, $\mathrm{Im}(x)$ and $\mathrm{Im}(y)$.\label{fig:disk}
}
\end{figure}

A different take on renormalization of one-dimensional analytic maps with Siegel disks was introduced by the third author in \cite{Ya1} based on the {\it cylinder renormalization operator} $\cren$. This operator acts on analytic maps defined in some neighborhood of a Siegel fixed point, rather than on pairs.
The definition of cylinder renormalization involves a non-linear, rather than linear rescaling of iterates. 
There exists a constant $s\in\NN$ such that the following holds.
Let $f$ be a cylinder-renormalizable analytic map $f$, and 
denote $(\eta,\xi)=\cR^{s-1}(f)$. Then the cylinder renormalization $\cren(f)$ is obtained by a non-linear rescaling
\begin{equation}
\label{eq:conj2}
\Phi\circ\eta\circ\Phi^{-1}=\cren(f)
\end{equation}
of the map $\eta$ by the uniformizing coordinate $\Phi$ 
of a particular fundamental domain (called a fundamental crescent in \cite{Ya1}) of the map $\xi$. Furthermore, (cf. Proposition~2.11 of \cite{Ya1}), the dependence
$$\xi\mapsto \Phi$$
is locally analytic.

For a topological disk $Z\ni 0$ denote $\cH(Z)$ the Banach space of holomorphic functions $f$ in $Z$ with the uniform norm, and set $\cH(Z,W)\equiv \cH(Z)\times\cH(W)$.
We will typically use the notation $(\eta,\xi)$ for an element of $\cH(Z,W)$.

We let $\cC(Z,W)$ denote the Banach subspace of $\cH(Z,W)$ given by the linear conditions 
$$\eta'(0)=\xi'(0)=0.$$
We say that a pair $(\eta,\xi)\in\cC(Z,W)$ is  {\it almost commuting to order $s\geq 0$} if the following holds:  
\begin{equation}\label{eq:accond}
(\eta \circ \xi)^{(n)}(0)=(\xi \circ \eta)^{(n)}(0), \ 0\leq n\leq s;\; \eta''(0)>0;\;\xi''(0)>0,\;\text{ and }\xi(0)=1.
\end{equation}
In the case $s=2$, we will simply call the pair {\it almost commuting (or a.c.)}.
We denote $\cB(Z,W)$ the subset of $\cC(Z,W)$ consisting of a.c. pairs.
In \cite{GaYa2}, it is shown that  there exists an open neighborhood $\cU$ of  $\cC(Z,W)$ such that $\cB(Z,W)\cap\cU$ is a Banach submanifold of  $\cH(Z,W)$.

Let $\theta$ be periodic under the Gauss map with period $p$, and denote $r_l=[1/G^l(\theta)]$ (these are the digits in the continued fraction expansion of $\theta$,
and $q_{n+1}=r_nq_n+q_{n-1}$). Similarly to the above, for a pair $\zeta=(\eta,\xi)$,  we define a sequence of pre-renormalizations
$$p\cR^n\zeta=\zeta_n=(\eta_n,\xi_n)$$
by $\zeta_0=\zeta$ and $\xi_{n+1}=\eta_n$, $\eta_{n+1}=\eta_n^{r_n}\circ\xi_n$. Renormalizations $\cR^n(\zeta)$ are then defined as
$$\cR^n(\zeta)=(\upsilon_n^{-1}\circ \eta_n\circ\upsilon_n,\upsilon_n^{-1}\circ \xi_n\circ\upsilon_n),\text{ where }\upsilon_n(z)=\xi_n(0)\cdot\kappa(z). $$

McMullen in \cite{Mc} showed that there exists a pair of analytic maps $\zeta_\lambda$ which is periodic under the action of  $\cR$ with period $p$, and such that 
for every $\lambda_1=e^{2\pi i\theta_1}$ where
$$G^m(\theta_1)=\theta\text{, for some }m\geq 0,$$
we have
$$\cR^{np+m}P_{\lambda_1}\to \zeta_\lambda\text{ at a rate, which is geometric in }n.$$

Let $\theta$ and $p$ be as above. Set 
\begin{equation}\label{k-def}
k=p\text{ if }p\text{ is even, or }k=2p\text{ if }p\text{ is odd},
\end{equation}
to guarantee that the operator $\cR^k$ is holomorphic (rather than anti-holomorphic).
Let us say that {\it renormalization hyperbolicity property} ({\bf H}) holds for  $\theta$ if the following is true:

\medskip
\noindent ({\bf H}) \ {\it There exist a pair of topological disks $\tl Z  \Supset Z, \quad \tl W \Supset W$ and $n=m k$, where $m \in \NN$ and $k$ is as in $(\ref{k-def})$ such that
\begin{itemize}
\item[$(i)$]   The operator $\cR^n$ is an analytic operator from an open neighborhood of its fixed point $\zeta_\lambda$ in $\cB(Z,W)$ to $\cB(\tl Z, \tl W)$. 
\item[$(ii)$]  The differential $D \cR^n  \arrowvert_{\zeta_\lambda}$ is a compact linear operator in $T_{\zeta_{\lambda}} \cB(Z,W)$. Let $M\equiv D {\cR^n} \arrowvert_{\zeta_\lambda}.$
 Then $M$ has a single simple eigenvalue outside of the closed unit disk, and the rest of the spectrum of $M$ lies inside the open unit disk.
\end{itemize}
}

We prove a conditional theorem:

\begin{thmx}\label{thm:C} Suppose renormalization hyperbolicity property {\rm ({\bf H})} holds for $\theta$, and let 
$\theta_1$ be such that $G^m(\theta_1)=\theta$ for some $m\in\NN$. Set
$\lambda_1=e^{2\pi i\theta_1}$. Then the following statements hold:
\begin{itemize}
\item[(I)] there exists $\delta>0$ such that if $|\mu|<\delta$ then the map $H_{\lambda_1,\mu}$ lies in the stable set of $\zeta_\lambda$; 
\item[(II)] every H{\'e}non-like map $H$ in $W^s(\zeta_\lambda)$ has a Siegel disk $\Delta_H$ whose boundary is a topological circle;
\item[(III)] the Carath{\'e}odory extension of the linearizing coordinate $\phi$ as in equation (\ref{eq:lin}) to a map $S^1\times\{0\}\to\partial\Delta_H$ is not $C^1$-smooth.
\end{itemize}
\end{thmx}

Our Theorems \ref{thm:A} and \ref{thm:B} will follow from Theorem \ref{thm:C} and the following statement proven in \cite{GaYa2}:

\medskip
\noindent
{\bf Golden-mean renormalization hyperbolicity} \cite{GaYa2}. {\it Renormalization hyperbolicity property} ({\bf H}) {\it holds for $\theta_*=(\sqrt{5}-1)/2.$ }

\medskip
\noindent

\section{Dynamical partitions and multi-indices}\label{sec:partition}
Consider the space $\cI$ of multi-indices $\bar s=(a_1,b_1,a_2,b_2,\ldots,a_m,b_m)$ where $a_j\in \NN$ for $2\leq m$, $a_1\in\NN\cup\{0\}$,
$b_j\in\NN$ for $1\leq j\leq m-1$, and $b_m\in\NN\cup\{0\}$. 
We introduce a partial ordering on multi-indices:
$\bar s\succ \bar t$ if $\bar s=(a_1,b_1,a_2,b_2,\ldots,a_m,b_m)$, $\bar t=(a_1,b_1,\ldots,a_k,b_k,c,d)$, where $k<m$ and 
either $c< a_{k+1}$ and $ d=0$ or $c=a_{k+1}$ and $d< b_{k+1}$.

For a pair of maps $\zeta=(\eta,\xi)$ and $\bar s $ as above we will denote 
$$\zeta^{\bar s}\equiv\xi^{b_m}\circ\eta^{a_m}\circ\cdots\circ\xi^{b_2}\circ\eta^{a_2}\circ\xi^{b_1}\circ \eta^{a_1}.$$
Similarly, 
$$\zeta^{-\bar s}\equiv (\zeta^{\bar s})^{-1}=(\eta^{a_1})^{-1}\circ(\xi^{b_1})^{-1}\circ\cdots\circ(\eta^{a_m})^{-1}\circ (\xi^{b_m})^{-1}.$$


Consider the $n$-th pre-renormalization of $\zeta$:
$$p\cR^n\zeta=\zeta_{n}=(\eta_{n}|_{Z_n},\xi_{n}|_{W_n}),$$
where $Z_{n} = \alpha_{n}(Z)$, $W_{n} = \alpha_{n}(W)$, and 
\begin{equation}\label{alpha_n}
\alpha_n(z)=\eta_n(0) z.
\end{equation}

We define  $\bar s_n,\bar t_n\in\cI$ to be such that
$$\eta_{n}=\zeta^{\bar s_n},\text{ and }\xi_{n}=\zeta^{\bar t_n}.$$
A straightforward induction shows:
\begin{lem}
\label{multi1}
Let $\bar r=\bar s_n$ or $\bar t_n$. Write $\bar r=(a_1,b_1,a_2,b_2,\ldots,a_{m_n},b_{m_n})$. Then $b_{m_n}=0$, and either
$$a_{m_n}\geq 2$$
or
$$a_{m_n}=b_{m_n-1}=1.$$
Furthermore, if $\bar s_n$ ends in $\ldots,1,1,0$ then so does $\bar t_n$.
\end{lem}

Let $\tau_{\theta}:\RR\to\RR$ be the translation $x\mapsto x+\theta$, with $\lambda=\exp(2\pi i\theta)$, and $\theta\in(0,1)$. Define
$$f(x)=\tau_{\theta}(x)\text{ and }g(x)=x-1,$$
and set
\begin{equation}
\label{rigid-pair}
I=[-1,0],\;J=[0,\theta],\text{ and }T=(f|_I,g|_J).
\end{equation}
Define $$T_n=(f_n,g_n)=(T^{\bar s_n},T^{\bar t_n}),$$ and set 
$$I_n=[0,g_n(0)],\;J_n=[0,f_n(0)]$$
(the notation $[a,b]$ denotes the interval with endpoints $a$, $b$, not necessarily in that order).

Now consider the collection of intervals
\begin{eqnarray}\label{eq:partition}
\cP_n\equiv \{T^{\bar w}(I_n)\text{ for all }\bar w\prec \bar s_n\text{ and }T^{\bar w}(J_n)\text{ for all }\bar w\prec \bar t_n\}.
\end{eqnarray}
It is easy to see that:
\begin{itemize}
\item[(a)] $\underset{H\in\cP_n}\cup H=I\cup J$;
\item[(b)] for any two distinct elements $H_1$ and $H_2$ of $\cP_n$, the interiors of $H_1$ and $H_2$ are disjoint.
\end{itemize}
In view of the above, we call $\cP_n$ the $n$-th dynamical partition of the segment $I\cup J$.

Consider the sequence of domains $$\cV_n\equiv \{\zeta^{\bar w}(Z_n)\text{ for all }\bar w\prec \bar s_n\text{ and }\zeta^{\bar w}(W_n)\text{ for all }\bar w\prec \bar t_n\}.$$
By analogy with the above definition (and somewhat abusing the notation) we call $\cV_n$ the $n$-th dynamical partition of the pair $\zeta$.

\begin{prop} \label{partition} Suppose, renormalization hyperbolicity property holds for $\theta$, and
$$\zeta\in W^s(\zeta_\lambda),\text{ where }\lambda=e^{2\pi i\theta}.$$
 Then there exists $N=N(\zeta)$, $K>0$,  and $0<\gamma<1$ so that for every $n>N$  the following properties hold.

\begin{itemize}
\item[1)] If $Q_n \in \cV_n$ then $\diam(Q_n)< \gamma^n$.
\item[2)] Any two neighboring domains $Q_n, Q'_n \in \cV_n$ are $K$-commensurable.
\item[3)] For every $\bar w \prec \bar s_n$ (or $\bar w \prec \bar t_n$)  set $\psi_{\bar w}^\zeta=\zeta^{\bar w} \alpha_n$. Then $\|D \psi_{\bar w}^\zeta |_Z \|_\infty<\gamma^n$ (or $\|D \psi_{\bar w}^\zeta |_W \|_\infty<\gamma^n$, respectively).

\end{itemize}
\end{prop}
\begin{proof}
By our assumption, there exists $N>0$ and a pair of domains $\hat Z\Supset Z$ and $\hat W\Supset W$  such that for all $n\geq N$ the 
maps of the pair $\cR^n\zeta\in\cC(\hat Z,\hat W)$.
By Koebe Distortion Theorem, this implies that for all $\bar w \prec \bar s_n$ (or $\bar w \prec \bar t_n$) the branches $\zeta^{-\bar w}$ have 
bounded distortion. The domain $Z_n=\alpha_n(Z)$ has diameter $O(\gamma^n)$. The claims readily follow.
\end{proof}

\medskip
\noindent

\section{Renormalization for  pairs of two-dimensional dissipative maps}\label{sec:RenACM}
This section contains a summary of the extention of the renormalization operator from the space $\cB(Z,W)$ of almost commuting pairs to an appropriately defined space of two-dimensional maps. The details of the procedure can be found in \cite{GaYa3}.

Let $\Omega, \Gamma$ be  domains in $\CC^2$. We denote $O(\Omega,\Gamma)$ the Banach space of bounded analytic functions 
$F=(F_1(x,y),F_2(x,y))$  from $\Omega$ and $\Gamma$ respectively to $\CC^2$ equipped with the norm
\begin{equation}
\label{eq:unormm}\| F\|= \frac{1}{2} \left(   \sup_{(x,y) \in \Omega}|F_1(x,y)|+  \sup_{(x,y)\in \Gamma} |F_2(x,y) | \right).
\end{equation}
 We let $O(\Omega,\Gamma,\delta)$ stand for the $\delta$-ball around the origin in this Banach space.

In what follows, we fix $W$, $Z$, $\tl Z$, and $\tl W$ as in ({\bf H}), and $R>0$ such that $\DD_R \subset Z \cap W$, and let 
$\Omega=Z \times \DD_R$, $\Gamma=W \times \DD_R$.
We select $\hat Z$ and $\hat W$ so that
$$Z \Subset \hat Z \Subset \tl Z,\; W \Subset \hat W \Subset \tl W.$$
We define an isometric embedding $\iota$ of the space $\cH(Z,W)$ into $O(\Omega,\Gamma)$ which send the pair $\zeta=(\eta,\xi)$ to
the  pair of functions $\iota(\zeta)$:
\begin{equation}
\label{eq:embed}
\left(\left(x \atop  y\right)\mapsto \left(\eta(x) \atop \eta(x) \right), 
\left(x \atop  y\right)\mapsto \left(\xi(x) \atop \xi(x) \right)   \right).
\end{equation}

Let $\cU$ be an open neighborhood of $\zeta_\lambda$ as in ({\bf H}) in $\cC(Z,W)$, and let $Q$ be a neighborhood of $0$ in $\CC$.
We will consider an open subset of  $O(\Omega,\Gamma)$ of pairs of maps of the form
\begin{eqnarray}
\label{eq:A} A(x,y)&=&(a(x,y),h(x,y))=(a_y(x),h_y(x)),\\
\label{eq:B} B(x,y)&=&(b(x,y),g(x,y))=(b_y(x),g_y(x)), 
\end{eqnarray}
such that
\begin{itemize}
\item[1)] the pair $(a(x,y),b(x,y))$ is in a $\delta$-neighborhood  of  $\cU$ in $O(\Omega,\Gamma)$,
\item[2)]  $(h,g) \in O(\Omega,\Gamma)$ are such that  $|\partial_x h(x,0)|>0$ and  $|\partial_x g(x,0)|>0$ whenever $x \notin \bar Q$,  and 
$$(h(x,y)-h(x,0),g(x,y)-g(x,0)) \in O(\Omega,\Gamma,\delta).$$ 
\end{itemize}
This open subset of $O(\Omega,\Gamma)$ will be denoted $\cA(\cU,Q,\delta)$ for brevity.

We say that a pair $(A,B)$ is a pre-renormalization of a map $H$:
$$(A,B)=p\cR^nH$$ if 
$$A=H^{q_n},\;B=H^{q_{n+1}}\text{ for some }n\geq 0.$$


\subsection{Defining renormalization: coordinate transformations}

Let $(\eta,\xi) \in \cB(Z,W)$ be $n\geq 2$ times renormalizable,  and consider its $n$-th pre-renormalization written as
$$p \cR^n\zeta=\left( \zeta^{\bar s_n}, \zeta^{\bar t_n } \right).$$
Let $\bar s_n$ be given by $(a_1,b_1,a_2,b_2,\ldots,a_{m_n},0)$ (recall Lemma~\ref{multi1}).
We denote 
\begin{eqnarray}
\nonumber \hat s_n&=&\left\{ (a_1,b_1,a_2,b_2,\ldots,a_{m_n}-2,0), \ a_{m_n} \ge 2  \atop (a_1,b_1,a_2,b_2,\ldots,0,0,0), \ a_{m_n}=1 \right. ,\\
\nonumber \phi_0(x)&=&\left\{ \eta^2, \ a_{m_n} \ge 2  \atop \eta \circ \xi, \ a_{m_n}=1 \right. .
\end{eqnarray}
Define $\hat t_{n}$ in an identical way to $\hat s_{n}$ (see Lemma~\ref{multi1}). Then $p \cR^n\zeta$ can be written as
$$ p \cR^n\zeta=\phi_0 \circ  \left(\zeta^{\h s_n}, \zeta^{\h t_n } \right).$$
 For a  sufficiently large $n$, the function $\eta^{-1}$ is a diffeomorphism of the neighborhood $\alpha_n (Z \cup W)$, and one can define the $n$-th pre-renormalization of $\zeta$ in $\eta^{-1}(\alpha_n (Z \cup W))$ as 
$$
\h p \cR^n\zeta=\left(\eta^{-1} \circ  \zeta^{\bar s_n} \circ \eta, \eta^{-1} \circ \zeta^{\bar t_n } \circ \eta \right)=\left( f \circ   \zeta^{\h s_n} \circ \eta, f \circ \zeta^{\h t_n } \circ \eta \right),
$$
where $f=\eta$ if $a_n \ge 2$ and $f=\xi$ if $a_n=1$.

Next, suppose that $\Sigma=(A,B)$ lies in $\cA(\cU,Q,\delta)$ with $\cU$ and $\delta$ sufficiently small, so that 
the following pre-renormalization  is defined in a neighborhood of $\eta^{-1}(\alpha_n (Z \cup W))\times \{ 0\}$:
$$\h p \cR^n \Sigma = \left(F \circ \Sigma^{\h s_n} \circ A, F \circ \Sigma^{\h t_n} \circ A \right),$$
where $F=A$ if $a_n \ge 2$ and $F=B$ if $a_n=1$.

We will denote $$\pi_1(x,y)=x\text{ and }\pi_2(x,y)=y.$$

Set
$$\phi_y(x)=\phi(x,y) :=\left\{ \pi_1 A^2(x,y), \ a_n \ge 2  \atop \pi_1 A \circ B(x,y), \ a_n=1 \right.$$

For sufficiently small $\delta$, the map $\phi_z$ is close to $\phi_0$ and is a diffeomorphism of a neighborhood of $\pi_1 \Sigma^{\h s_n}(\alpha_n(Z),0) \approx \zeta^{\h s_n}(\alpha_n (Z))$ for all $z \in \DD_R$ for some  $R=R(\delta)>0$.  Similarly, $g_z$ is a diffemorphism of  a neighborhood of $\pi_1  \Sigma^{\h s_n}(\alpha_n (Z),0)$ for all $z \in \DD_R$ for some $R=R(\delta)>0$.

Furthermore, set
$$q_z(x) \equiv q(x,z)=\pi_2 F(x,z)= \left\{ h_z(x), \ a_n \ge 2  \atop g_z(x), \ a_n=1 \right.$$
According to our definition of the class  $\cA(\cU,Q,\delta)$, this is a diffeomorphism outside a neighborhood of zero. Also, set
\begin{equation}
\nonumber w_z(x) \equiv w(x,z) := q_{z}\left(\phi_{z}^{-1}( x) \right),
\end{equation} 
a diffeomorphism of a neighborhood of   $\pi_1 \phi_z \circ \Sigma^{\h s_n}(\alpha_n (Z),0)$ in $\CC^2$ onto its image for all $z \in \DD_R$ for some  $R=R(\delta)>0$.  Notice, that $\partial_z  w_z(x)$ and $\partial_z w^{-1}_z(x)$ are functions whose uniform norms are $O(\delta)$.

\noindent
Define the following transformation:
\begin{equation}\label{eq:Htransform}
H_{\Sigma}(x,y) = (a_y(x),w_{q_0^{-1}(y)}^{-1}(y)),
\end{equation}
This transformation is $\delta$-close to $\left(\eta(x),\phi_0(q_0^{-1}(y)) \right)$ in $O(\Omega,\Gamma)$, and therefore, for small $\delta$,  is a diffieomorphism of a neighborhood of $\pi_1  F \circ \Sigma^{\h s_n}(\alpha_n (Z),0) \approx f(\zeta^{\h s_n}(\alpha_n (Z)))$ onto its image.  In particular, 
\begin{equation}
\label{eq:AHinv} A \circ H_{\Sigma}^{-1}(x,y)=(x,h(\eta^{-1}(x),y))+O(\delta).
\end{equation}

We use $H_{\Sigma}(x,y)$ to pull back $\hat p \cR^n \Sigma$ to a neighborhood of definition of the $n$-th pre-renormalization of a pair $(\eta,\xi)$ - that is, a neighborhood of  $\alpha_n (Z \cup W)$ in $\CC^2$:
$$p \cR^n \Sigma=(\bar A, \bar B) = H_{\Sigma}\circ F \circ \left( \Sigma^{\h s_n},  \Sigma^{\h t_n} \right) \circ A \circ H_{\Sigma}^{-1}(x,y).$$

The following has been proved in \cite{GaYa2}.
\begin{lem}
\label{lem-preren}
There exists an $n \in \NN$, and a choice of $\cU$, $Q$, $\delta_0$ and $C>0$ such that the following holds. For every $\delta<\delta_0$ and every $\Sigma\in\cA(\cU,Q,\delta)$ the pair $p\cR^n\Sigma$ is defined, lies in $O(\hat\Omega,\hat\Gamma)$, $\hat \Omega= \hat Z \times \DD_R$,  $\hat \Gamma= \hat W \times \DD_R$, and 
$$\dist(p \cR^n \Sigma, \iota(\cH(\alpha_n(\hat Z),\alpha_n(\hat W))))<{C \delta (\|  \pi_1 \Sigma-\pi_2 \Sigma\|+\delta   )}.$$
\end{lem}

Let us write
\begin{equation}
\bar{A}(x,y)=\left( { \bar \eta_1(x)+\bar \tau_1(x,y) \atop   \bar \eta_2(x)+\bar \tau_2(x,y) } \right),
\end{equation}
where 
$$\bar \eta_1(x) \equiv \pi_1 \bar{A}(x,0), \quad  \bar \eta_2(x) \equiv \pi_2 \bar{A}(x,0)$$
are $O({ \delta \|  \pi_1 \Sigma-\pi_2 \Sigma\|+\delta^{2}   })$-close to each other, and both are $\delta$-close to $\pi_\eta p \cR^n \zeta=\zeta^{\bar s_n}$, where  $\pi_\eta$ and $\pi_\xi$ are the projections on, correspondingly, the first and the second map in a pair, and 
$$\bar \tau_1(x,y) \equiv  \pi_1 \bar{A}(x,y)- \pi_1 \bar{A}(x,0), \quad \bar \tau_2(x,y)= \pi_2 \bar{A}(x,y)- \pi_2 \bar{A}(x,0),$$
are functions whose norms are $O(\delta^2)$. Similarly,
$$
\nonumber \bar{B}(x,y)=\left( { \bar \xi_1(x)+\bar \pi_1x,y) \atop   \bar \xi_2(x)+\bar \pi_2(x,y) } \right),
$$
where
$$\bar \xi_1(x) \equiv \pi_1 \bar{B}(x,0), \quad  \bar \xi_2(x) \equiv \pi_2 \bar{B}(x,0)$$
are $O({\delta \|  \pi_1 \Sigma-\pi_2 \Sigma\|+\delta^{2}   })$-close to each other, and both are $\delta$-close to $\pi_\xi p \cR^n \zeta=\zeta^{\bar t_n}$, and 
$$\bar \pi_1(x,y) \equiv  \pi_1 \bar{B}(x,y)- \pi_1 \bar{B}(x,0), \quad \bar \pi_2(x,y)= \pi_2 \bar{B}(x,y)- \pi_2 \bar{B}(x,0),$$
are functions whose norms are $O(\delta^2)$.

\subsection{Defining renormalization: critical projection} 
 By the Argument Principle, 
if $\delta$ is sufficiently small, then 
the function $\pi_1 \bar B \circ \bar A (x,0)$ has a unique critical point $c_1$ in a neighborhood of $0$. Set $T_1(x,y)=(x+c_1,y)$, then
$$\partial_x \left(\pi_1 T_1^{-1} \circ \bar B \circ \bar A \circ T_1  \right)(0,0)=0.$$ Similarly, if $\delta$ is sufficiently small, the function $\pi_1 T_1^{-1} \circ \bar A \circ \bar B \circ T_1(x,0)$ has a unique critical point $c_2$ in a neighborhood of $0$.  Set $T_2(x,y)=(x+c_2,y)$, then
$$\partial_x \left(\pi_1 T_2^{-1} \circ T_1^{-1} \circ \bar A \circ \bar B \circ T_1  \circ T_2 \right)(0,0)=0.$$

We now set
\begin{eqnarray}
\nonumber \Pi_1(\bar A, \bar B)&=&(\tilde A, \tilde B):=(T_2^{-1} \circ T_1^{-1} \circ \bar A \circ T_1, T^{-1}_1 \circ \bar B \circ T_1 \circ T_2)\\
\nonumber &=& \left( \left( \tilde \eta_1(x) + \tilde \tau_1(x,y) \atop   \tilde \eta_2(x) + \tilde \tau_2(x,y)  \right), \left( \tilde \xi_1(x) + \tilde \pi_1(x,y) \atop   \tilde \xi_2(x) + \tilde \pi_2(x,y)  \right)   \right),
\end{eqnarray}
where the norms of the functions $\tilde \tau_k$, $\tilde \pi_k$, $k=1,2$, are $O(\delta^2)$. 

The critical points  of the functions $\pi_1 (\bar A \circ \bar B)(x,0)$ and $\pi_1 (\bar B \circ \bar A)(x,0)$ are $O\left( \delta \|\pi_1 \Sigma -\pi_2 \Sigma \|  +\delta^{2} \right)$-close to each other, and therefore, 
\begin{equation}
\label{T2small} T_2=\text{Id}+O\left( \delta  \|\pi_1 \Sigma -\pi_2 \Sigma \|  +\delta^{2} \right).
\end{equation}

Let us set 
$$\tl\Sigma=(\tl A,\tl B)=\Pi_1 p \cR^n\Sigma.$$
We note that if the maps $\bar A$ and $\bar B$ commute, than the critical point of $\pi_1 T_1^{-1} \circ \bar A \circ \bar B \circ T_1(x,0)$ is at $0$. We, therefore, have the following
\begin{prop}
Suppose $(A,B)$ is a pre-renormalization of a map $H$. Then the map $T_2\equiv \text{Id}$, and hence, the projection $\Pi_1$ is a conjugacy by $T_1$.
\end{prop}

\subsection{Defining renormalization: commutation projection} At the next step we will project the pair $(\tilde A, \tilde B)$ onto the subset of pairs satisfying the following almost commutation conditions:
\begin{eqnarray} \label{commutation}
\partial_x^i \pi_1(\tilde A \circ \tilde B(x,0) - \tilde B \circ \tilde A(x,0))\arrowvert_{x=0}&=&0, \quad i=0,2 \\
\label{normalization} \pi_1 \tilde B(0,0)&=&1.
\end{eqnarray}
To that end we set 
$$\Pi_2 (\tilde A,\tilde B)(x,y)=\left( \left( \tilde \eta_1(x) +a x^4+b x^6 + \tilde \tau_1(x,y) \atop   \tilde \eta_2(x) +a x^4+b x^6+ \tilde \tau_2(x,y)  \right), \left( \tilde \xi_1(x)+c  + \tilde \pi_1(x,y) \atop   \tilde \xi_2(x) +c +\tilde \pi_2(x,y)  \right)   \right),$$
and require that $(\ref{commutation})$ and $(\ref{normalization})$ are satisfied for maps in the pair $\Pi_2 (\tilde A,\tilde B)(x,y)$. The following Proposition is proved in \cite{GaYa2}.
\begin{prop}\label{prop:2Dprojection}
There exists $\rho>0$ such that for all $\tl\Sigma$ in the $\rho$-neighborhood of $$\iota(\cC(\alpha_n(\hat Z),\alpha_n(\hat W)))$$
there is a unique tuple $(a,b,c,d)$ such that the pair $ \Pi_2 (\tilde A,\tilde B)$ satisfies the equations $(\ref{commutation})$ and $(\ref{normalization})$. Moreover, in this neighborhood, the dependence of $\Pi_2$ on $\Sigma$ is analytic.  
Furthermore, if $A\circ B=B\circ A$, then 
$\Pi_2=\text{Id}.$

\end{prop}

\noindent
Let us fix $n\in 2\NN$, $\cU$, $Q$, $\delta$ so that Lemma~\ref{lem-preren}  holds, and furthermore, the image 
$\Pi_1p\cR^n\cA(\cU,Q,\delta)$ lies in the $\rho$-neighborhood of $\iota(\cC(\alpha_n(\hat Z),\alpha_n(\hat W)))$ as in Proposition~\ref{prop:2Dprojection}. We then have:

\begin{prop}
\label{small-pert2}
For every $\Sigma \in \cA(\cU,Q,\delta)$,
$${\rm dist}(\Pi_2\Pi_1 p \cR^n \Sigma, \iota(\cB(\alpha_n(\hat Z),\alpha_n(\hat W)))) < {C { \delta (\|  \pi_1 \Sigma-\pi_2 \Sigma\|+\delta   )}}.$$
\end{prop}

Let $\ell_n=\pi_1\bar B(0,0)$ and  $\Lambda_n(x,y)=(\ell_n x, \ell_n y)$.
\begin{defn}
 We define {\it the  renormalization of depth $n$ } of a pair $\Sigma \in \cA(\cU,Q,\delta)$ as
\begin{equation}\label{2Drenorm}
\cR_n\Sigma = \Lambda_n^{-1} \circ\Pi_2\circ\Pi_1\circ p\cR^n \Sigma \circ \Lambda_n.
\end{equation}
Given a map  $H$ from a subset of $\CC^2$ to $\CC^2$, such that the pair $(A,B)=p\cR^nH=(H^{q_{n+1}},H^{q_n}) \in  \cA(\cU,Q,\delta)$ for some integer $n$, we will also use the shorthand notation 
\[
\cR_n H \equiv \Lambda_n^{-1} \circ\Pi_2\circ\Pi_1\circ p\cR^n H \circ \Lambda_n.
\]
\end{defn}

\subsection{Hyperbolicity of renormalization of 2D dissipative maps}
We conclude this section by formulating the following theorem:
\begin{thm}
\label{mainthm6}
Given a $p$-periodic $\theta$, set that $\lambda=e^{2 \pi i\theta}$. Assume that {\rm ({\bf H})} holds.  Then there exists an even $n=m k$, where $m \in \NN$ and $k$ is as in $(\ref{k-def})$, such that the point $\iota(\zeta_\lambda)$ is a fixed point of $\cR_n$ in $O(\Omega,\Gamma)$. The linear operator
$N=D \cR_n \arrowvert_{\iota(\zeta_\lambda)}$  is compact. The spectrum of $N$ coincides with the spectrum of $M$, where $M$ is as in {\rm ({\bf H})}. More specifically, $\kappa \neq 0$ is an eigenvalue of $M$, and $h$ is a corresponding eigenvector if and only if $\kappa$ is an eigenvalue of $N$, and $D\iota(h)$ is a corresponding eigenvector.
\end{thm}
\begin{proof}
Since $\iota$ is an immersion on $\cC(Z,W)$, and 
$$\iota\circ \cR^k=\cR_k\circ\iota, $$
the spectral decomposition of $N$ splits into the direct sum $T_1\oplus T_2$, where $T_1$ is the tangent subspace
$$T_1=T_{\iota(\zeta_\lambda)}\iota(\cB(Z,W)).$$
The restriction $N|_{T_1}$ is isomorphic to $M$. Further, by 
Proposition~\ref{small-pert2}, the magnitude of a perturbation of $\iota(\zeta_\lambda)$ in the direction of a vector in $T_2$ is decreased quadratically by 
$(\cR_n)^2$. Hence, in the spectral decomposition, the subspace $T_2$ corresponds to the zero eigenvalue.
\end{proof}

\medskip
\noindent
\section{Proof of Theorem C.}
\subsection{The H{\'e}non family intersects $W^s(\zeta_\lambda)$.}
Let us fix $\theta$, $\theta_1$, $\lambda$, $\lambda_1$ as in Theorem \ref{thm:C}. As before, let $k$ be as in $(\ref{k-def})$, and let $n$ be as in Theorem $\ref{mainthm6}$. For brevity, in what follows, we set 
\begin{equation} \label{RG} 
\cRG=\cR_n.
\end{equation}
We prove:
\begin{thm}
\label{th:intersect}
 There exists $\delta>0$ such that if $|\mu|<\delta$ then the one-parameter family  $l\to H_{l,\mu}$ intersects the stable set of $\zeta_\lambda$ under $\cRG$. 
\end{thm}
\begin{proof}
Let $U\ni 0$ be a Jordan domain in $\CC$ and let $\bC_U$ denote the Banach space of bounded analytic maps $f$ in $U$ equiped with a uniform norm $\| \cdot \|_U$ and such that $f(0)=0$.
Let $f_*$ be the periodic  point of $\cren$ with $f_*'(0)=e^{2\pi i\theta}$ constructed in \cite{Ya1}. 
We denote the period of $f_*$ under $\cren$ by $p$.  As shown in \cite{Ya1}, there exists a choice of domains $U_1\Supset U$ such that 
$$f_*\in\bC_U\text{ and }\cren f_*\in\bC_{U_1}.$$ 

Let $n$ be as in  Theorem $\ref{mainthm6}$. For the quadratic polynomial $P_{\lambda_1}$ there exists $N$ such that its $N n$-th cylinder renormalization lies in the local stable set  of $f_*$ in $\bC_U$. 

As is shown in \cite{Ya1}, the family $l\mapsto \cren^N P_l$ lies in the unstable cone field of $\cren$. Specifically, if 
$$l_t=\lambda+t,$$
then
\begin{equation}
\label{eq1} ||\cren^{(i+N) n}P_{l_t}-\cren^{(i+N) n}P_\lambda ||_U=a\beta^it+o(t),\text{ where }\beta>1\text{ and } a>0.
\end{equation}
Let us select $i$ large enough, so that $\cren^{(i+N) n }P_\lambda \in \bC_{U_2}$ with $U_2\Supset U$. By Koebe Distortion Theorem, 
\begin{equation}
\label{eq2} ||\cren^{(i+N) n }P_{l_t}-\cren^{(i+N) n }P_\lambda ||_U\sim |(\cren^{(i+N) n}P_{l_t})(1)-(\cren^{(i+N) n}P_\lambda)(1)|,
\end{equation}
where  $1$ is the critical point.

Let us turn to renormalization of  commuting pairs. We recall that, according to (\ref{eq:conj2}), $sn$ steps of $\cR $ correspond to $n$ steps of the operator $\cren$.  
Using Koebe Distortion Theorem again, we see that
\begin{equation}
\label{eq3} 
||\cR^{(i+N)sn}P_{l_t}-\cR^{(i+N)sn}P_\lambda||\sim |(\cR^{(i+N)sn}P_{l_t})(0)-(\cR^{(i+N)sn}P_\lambda)(0)|.
\end{equation}
Denote $$(\eta_l,\xi_l)=\cR^{s-1}(\cren^{(i+N) n-1}P_l).$$
Let $\Phi_t$, $\Phi_0$ denote the uniformizing coordinates of the fundamental crescents of $\xi_{l_t}$, $\xi_\lambda$ respectively (\ref{eq:conj2}).
Note that, by complex {\it a priori} bounds \cite{Ya1} and Koebe Distortion Theorem, $\Phi_{t}$ has universally bounded distortion and $\Phi_{t}'\simeq 1$.
We have
\begin{equation}
\label{eq3a}
||\Phi_{{t}}-\Phi_{0}||\sim ||\cR^{(i+N)sn}P_{l_t}-\cR^{(i+N)sn}P_\lambda||.
\end{equation} 
The estimates (\ref{eq1})-(\ref{eq3a}) imply that 
$$||\cR^{(i+N)sn}P_{l_t}-\cR^{(i+N)sn}P_\lambda||\sim \beta^it.$$
Thus the family
$$l\mapsto g_l\equiv \iota\cR^{Nsn}P_{l}$$ lies in the expanding cone field of $\zeta_\lambda$ under $\cRG$.

\begin{figure}[htb]
\begin{center}
{\includegraphics[width=\textwidth]{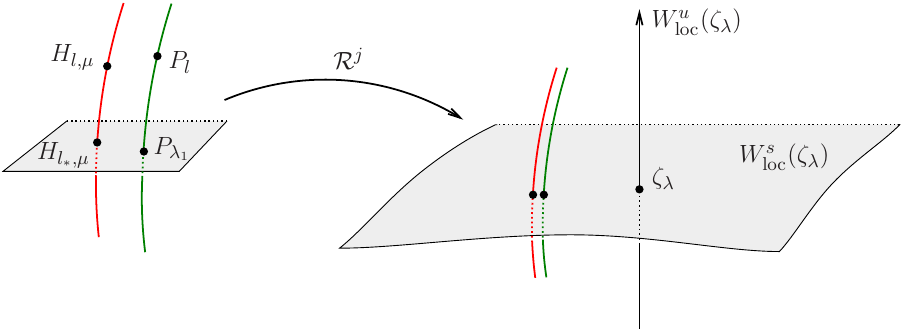}}
\end{center}
\caption{An illustration to the proof of Theorem~\ref{th:intersect}; $j=Nns$.
\label{fig:transverse}}
\end{figure}

\noindent Since for a small enough $\mu$, the family
$$l\mapsto G_l\equiv \cRG^{N s}H_{l,\mu}$$
is a $C^1$-small perturbation of $g_l$, it is transverse to $W^s_\text{loc}(\zeta_\lambda)$ and hence, intersects with it (see Fig.~\ref{fig:transverse}).
\end{proof}


\subsection{Construction of an invariant curve}
In this section we prove the following statement:
\begin{prop}
\label{prop-curve}
There exists $\eps>0$ such that the following holds. Let $|\mu|<\eps$, and 
$$H_{l_{*},\mu}\in W^s(\zeta_\lambda)\text{ where }\lambda=e^{2\pi i\theta}.$$
Denote $\Omega_n$, $\Gamma_n$ the domains of definition of the $n$-th pre-renormalization $p\cR^n H_{l_*,\mu}$. 
Then
there exists a curve $\gamma_*\subset \CC^2$ such that the following properties hold:
\begin{itemize}
\item $\gamma_*$ is a homeomorphic image of the circle;
\item $\gamma_*\cap \Omega_n\neq \emptyset$ and $\gamma_*\cap \Gamma_n\neq \emptyset$ for all $n\geq 0$;
\item there exists a topological conjugacy 
$$\varphi_*:\TT\to \gamma_*$$
between the rigid rotation $x\mapsto x+\theta_1\mod \ZZ$ and $H_{l_{*},\mu}|_{\gamma_*}$;
\item there exists $m$ such that $G^m(\theta_1)=\theta$;
\item the conjugacy $\varphi_*$ is not $C^1$-smooth.
\end{itemize}

\end{prop}

\noindent
Before proving the above Proposition, we need to introduce some further notation.
Below, for brevity, we will denote $\Upsilon^1=\Omega,  \Upsilon^2=\Gamma$.

We set $n=k m$, as in Theorem $\ref{mainthm6}$ for some $m \ge 1$ (to be fixed later). 

To differentiate between transformations for different pairs we will use the folllowing notation.
Denote
$$\bar s_n=(a_1,b_1,\ldots,a_{m_n},0)\text{ and }\bar t_n=(c_1,d_1,\ldots,c_{l_n},0).$$
Given a pair $\Sigma$, denote $\Lambda_\Sigma$ the rescaling that corresponds to the first renormalization $\cRG$, and $H_\Sigma$ - the transformation constructed for $\Sigma$ in \eqref{eq:Htransform}, that is
$$\cRG \Sigma=   \Lambda_\Sigma^{-1} \circ T_\Sigma^{-1}\circ H_\Sigma\circ \left( \Sigma^{\tilde s_n}, \Sigma^{\tilde t_n} \right)\circ H_\Sigma^{-1}\circ T_\Sigma  \circ \Lambda_\Sigma = L_\Sigma^{-1} \circ \hat p \cR^n \Sigma \circ L_\Sigma,$$
where
\begin{equation}\label{sprim}
{\tilde s_n}=(1,0,a_1,b_1, \ldots, a_{m_n}-1,0)\text{, and }{\tilde t_n}=(1,0,c_1,d_1, \ldots, c_{l_n}-1,0),
\end{equation}
and 
$$L_\Sigma=  H_\Sigma^{-1}\circ T_\Sigma \circ \Lambda_\Sigma.$$
Note that since the elements of $\Sigma$ commute, the projection $\Pi_2=\text{Id}$ and $\Pi_1$ is the conjugation by the translation $T_\Sigma:=T_1$. 


Let $\bar s_n^l$ and $\bar t_n^l$ be defined by
$$(\hat p\cR^n)^l\zeta=(\zeta^{\bar s_n^l},\zeta^{\bar t_n^l}).$$ 

For each  multi-index $$\bar w=(a_0,b_0,a_1,b_1, \ldots, a_k,b_k) \prec \bar s_{n}^l\text{ or }\bar w=(a_1,b_1, \ldots, a_k,b_k) \prec \bar t_{n}^l$$ we define a domain
\begin{equation} \label{Qiw}
 Q_{\bar w}^i=\Sigma^{\bar w} \circ  L_\Sigma \circ L_{\cRG  \Sigma} \circ \ldots \circ L_{\cRG^{l-1} \Sigma}(\Upsilon^i), \ i=1 \text{ for }\bar w  \prec \bar s_{n}^l, \ i=2 \text{ for }\bar w  \prec \bar t_{n}^l.
\end{equation}

\noindent
By analogy with a dynamical partition of a commuting pair from  Section \ref{sec:partition}, the collection 
$$\cQ_{ln}\equiv \{ Q_{\bar w}^i\}$$ will be refered to as the $l n$-th partition for the two-dimensional  pair $\Sigma$.

Given $\Sigma \in W^s_{\text{loc}}(\zeta_\lambda)$, consider the following collection of  functions defined on $\Omega \cup \Gamma$:
$$\Psi_{\bar w}^\Sigma= \Sigma^{\bar w} \circ L_\Sigma.$$

\noindent
Given a collection of index sets $\{\bar w^i\}$, $\bar w^i \prec \bar s_{n}$ or $\bar w^i \prec \bar t_{n}$, consider the following {\it renormalization microscope}
$$\Phi_{\bar w^0,\bar w^1, \bar w^2, \ldots, \bar w^{j-1},\Sigma}^{j}= \Psi_{\bar w^0}^\Sigma \circ \Psi_{\bar w^1}^{\cRG \Sigma} \circ \ldots \circ  \Psi_{\bar w^{j-1}}^{\cRG^{(j-1)} \Sigma},$$
which we will also denote $\Phi^j_{{\hat w_0^{j-1}},\Sigma}$, where $\hat w_0^{j-1}=\left\{\bar w^0,\bar w^1, \bar w^2, \ldots, \bar w^{j-1} \right\}$,  for brevity. 

\begin{lem} The renormalization microscope  maps a set $\Upsilon^i$ onto an element of partition $\cQ_{j n}$ for $\Sigma$.
\end{lem}
\begin{proof}
The claim holds for $j=1$ by the definition $(\ref{Qiw})$ of the elements of the partition.

Assume that it $\Phi^j_{\hat w_0^{j},\Sigma}(\Upsilon^i)$ is an element of partition $\cQ_{j n}$ for $\Sigma$.

Consider  $\Phi^{j+1}_{\hat w_0^j,\Sigma}(\Upsilon^i)$:
$$\Phi^{j+1}_{\hat w_0^{j},\Sigma}(\Upsilon^i)=\Psi_{\bar w^0}^\Sigma \circ \Psi_{\bar w^1}^{\cRG \Sigma} \circ \ldots \circ  \Psi_{\bar w^j}^{\cRG^{j } \Sigma}(\Upsilon^i).$$
By assumption, 
$$\Phi_{\hat w_1^{j},\cRG \Sigma}^j(\Upsilon^i) \equiv \Psi_{\bar w^1}^{\cRG \Sigma} \circ \ldots \circ  \Psi_{\bar w^j}^{\cRG^{j } \Sigma}(\Upsilon^i)$$ 
is an element of the partition of level $j n$ for the pair $\cRG \Sigma$, that is, by $(\ref{Qiw})$  
$$\Phi_{\hat w_1^{j},\cRG \Sigma}^j(\Upsilon^i)=(\cRG \Sigma)^{\bar v}  \circ  L_{\cRG \Sigma} \circ L_{\cRG^2 \Sigma} \circ \ldots \circ L_{\cRG^j \Sigma} (\Upsilon^i),$$
for some admissible $\bar v=(\alpha_0,\beta_0, \alpha_1, \beta_1, \ldots, \alpha_m,\beta_m)$. Therefore, using the shorthand $$\cRG \Sigma=(A_1,B_1),$$ we have:
\begin{eqnarray}
\nonumber \Phi^{j+1}_{\hat w_0^{j},\Sigma}(\Upsilon^i)&=& \Psi_{\bar w^0}^\Sigma \circ \Phi_{\hat w_1^{j},\cRG \Sigma}^j(\Upsilon^i), \\
\nonumber &=& \Sigma^{\bar w^0} \circ L_{\Sigma} \circ (\cRG \Sigma)^{\bar v} \circ  L_{\cRG \Sigma}  \circ \ldots \circ L_{\cRG^j \Sigma} (\Upsilon^i)  \\
\nonumber &=& \Sigma^{\bar w^0} \circ L_{\Sigma} \circ ( B_1^{\beta_m} \circ A_1^{\alpha_m} \circ \ldots \circ  B_1^{\beta_0} \circ A_1^{\alpha_0}) \circ  L_{\cRG \Sigma}  \circ \ldots \circ L_{\cRG^j \Sigma} (\Upsilon^i)  \\
\nonumber &=& \Sigma^{\bar w^0} \circ L_{\Sigma} \circ \Lambda_\Sigma^{-1} \circ H_\Sigma \circ T_\Sigma^{-1}\circ \hspace{-1.0mm}\left( \!\left(\Sigma^{\tilde t_n}\right)^{\beta_m}\hspace{-1.5mm} \circ \!\left(\Sigma^{\tilde s_n}\right)^{\alpha_m}\hspace{-1.5mm} \circ\! \ldots \!\circ \!  \left(\Sigma^{\tilde t_n}\right)^{\beta_0}\hspace{-1.5mm} \circ \! \left(\Sigma^{\tilde s_n}\right)^{\alpha_0} \!\right)\! \circ \\
 \nonumber &\phantom{=}&  \phantom{\Sigma^{\bar w^0}} \circ T_\Sigma \circ H_{\Sigma}^{-1} \circ \Lambda_\Sigma \circ  L_{\cRG \Sigma}  \circ \ldots \circ L_{\cRG^j \Sigma} (\Upsilon^i)  \\
\nonumber &=& \Sigma^{\bar u} \circ L_\Sigma  \circ \ldots \circ L_{\cRG^j \Sigma} (\Upsilon^i),
\end{eqnarray}
for some index $\bar u$. By  (\ref{Qiw}), the latter is an element of the partititon $\cQ_{(j+1) n}$.
\end{proof}


\noindent
Since ${\cRG^{l} \Sigma}$ converges to $\zeta_\lambda$ at a geometric rate, the function $\Psi_{\bar w}^{\cRG^l \Sigma}$  converges to the function $(\psi_{\bar w}^{\zeta_*}, \psi_{\bar w}^{\zeta_*})$, defined in Proposition~\ref{partition}, at a geometric rate in $C^1$-metric. Therefore, by  Proposition~\ref{partition}, there  exists a neighborhood $\cS$ in $W^s_{\text{loc}}(\zeta_\lambda)$, and a sufficiently large $l$, such that 
$$\|D \Psi_{\bar w}^{\cRG^l \Sigma} |_{\Upsilon^i} \|_\infty < {1 \over 2},$$
whenever $\cRG^l \Sigma \in \cS$.

For every $\Sigma \in W^s_{\text{loc}}(\zeta_\lambda)$,  there exists $i_0\in\NN$ such that $\cRG^{i} \Sigma \in \cS$ for $i\geq i_0$. Hence,  there exists $C=C(\Sigma)$, such that
\begin{equation}
\label{miccontract}
\|D \Phi^j_{{\hat w}, \Sigma}|_{\Upsilon^i} \|_\infty < {C \over 2^j},
\end{equation}
and thus the renormalization microscope is a uniform metric contraction.

We are now ready to prove \propref{prop-curve}. 
\begin{proof}[Proof of \propref{prop-curve}]
Let $$\cRG^r(H_{l_{*},\mu})\equiv \Sigma=(A, B)\in W^s(\zeta_\lambda)$$
for some $r \in \NN$.

Select a distinct point $(x_{\bar w},y_{\bar w})$ in each of the sets $Q^i_{\bar w} \in \cQ_{l n}$. Consider the $l n$-th dynamical partition $\cP_{l n}$ for the pair $T$ as defined in (\ref{rigid-pair}).    Consider a piecewise-constant map  $\varphi_l$ sending the element of the partition with a multi-index $\bar w$ to $(x_{\bar w},y_{\bar w})$. According to (\ref{miccontract}), the diameters of the sets $Q^i_{\bar w}$ decrease at a geometric rate. Thus,  the maps $\varphi_l$  converge uniformly to a continuous map $\varphi$ of the interval $[-1,\theta]$ which is a homeomorphism onto the image. Set 
$$\varphi([-1,\theta])\equiv \gamma.$$
By construction,
$$\varphi\circ T=\Sigma \circ \varphi.$$

Let $\gamma_1\subset K^+(H_{l_{*},\mu})$ be the preimage of $\gamma$ under renormalization rescaling, and set
 $$\gamma_*\equiv \cup_{n\in\NN}H_{l_{*},\mu}(\gamma_1).$$
The conjugacy $\varphi$ induces a conjugacy 
$$\varphi_*:\TT\to \gamma_*$$
between a rigid rotation and $H_{l_{*},\mu}|_{\gamma_*}$. Hence, setting $l_{*}=e^{2\pi i\theta_1}$, we have
$$G^r(\theta_1)=\theta$$
for some $r \geq 0$. 

Finally, since the limiting pair $\zeta_\lambda$ has a critical point at $z=0$, the conjugacies $\varphi$ and $\varphi_*$ cannot be $C^1$-smooth. Indeed, assume the contrary. This would imply that there exists $K>1$ such that for every arc $J\subset \gamma_*$ and every $n\in\NN$, we have
\begin{equation}
\label{eq:contradiction}
\frac{1}{K}\diam(J)<\diam(H_{l_{*},\mu}^n(J))<K\diam(J).
\end{equation}
However, let $\Omega_n$, $\Gamma_n$ denote the domains of the pair $p\cR^nH_{l_{*},\mu}$. Let $z\in\gamma_*\cap \Omega_n$ and $z'=H_{l_{*},\mu}^{q_{n}}(z)$,
and denote $J_n$ the smaller subarc of $\gamma_*$ bounded by these two points. Since 
$$\cRG H_{l_{*},\mu}\approx \zeta_\lambda$$
for large values of $n$, we have
$$\diam (H_{l_{*},\mu}^{q_{n+1}}(J_n))\sim \left( \diam(J_n) \right)^2.$$
This clearly contradicts (\ref{eq:contradiction}).
\end{proof}

\subsection{The curve $\gamma_*$ bounds a Siegel disk}
Let us define a {\it $\varrho$-vertical cone field} in the tangent bundle $T\Omega$ where $\Omega$ is a subdomain of $\CC^2$ as
\[
C^{\text{vert},\varrho}_{(x,y)}=\left\{ (u,v)\in T_{(x,y)} \Omega,\ |u|<\varrho|v|\right\}.
\]


Let $f:U \rightarrow \CC$ be a holomorphic map. We consider two-dimensional perturbations of this map $F:\Omega\rightarrow \CC^{2}$ of the form
\begin{equation}
\label{2Dmap}
F(x,y)=\left(w(x,y),h(x,y) \right)=\left(f(x)+\tau(x,y),g(x)+ \chi(x,y) \right).
\end{equation}

We note:
\begin{prop}\label{prop:cone-inv} Suppose $|f'(x)|, |g'(x)|>\kappa$ and $|f'(x)|,|g'(x)|<K$ on the domain $U$ for some $\kappa>0$. Let $F^{-1}$ be defined on $\Delta=F(\Omega)$. 

Then there exist $\eps>0$ and $\varrho>0$ such that the following holds. Suppose the uniform norms of $\tau$ and $\chi$ in (\ref{2Dmap}) on $\Omega$ are bounded by $\eps$.  Given $\hat{\Delta} \Subset \Delta$, for every $(x,y)\in \hat{\Delta}$, denoting $(x_1,y_1)=F(x,y)$, we have
\[
D F^{-1}\big{|}_{(x_1,y_1)}\left(C^{\text{vert},\varrho}_{(x_1,y_1)}\right)\subset \ C^{\text{vert},\varrho}_{(x,y)},
\]
and $\|D F^{-1}\| >O\left({\kappa \over K \eps }\right)$  in $C^{\text{vert},\varrho}$.
\end{prop}
\begin{proof}
Let $w_i(x,y)=\partial_i w(x,y)$ denote the $i$-th partial derivative of $w(x,y)$, $i=1,2$. Similarly for $h_i(x,y)$. A simple computation shows that:
\begin{equation}D F^{-1}(x_1,y_1) \left[ u \atop v\right]={ 1 \over D(x,y)} \left[\phantom{-} h_2(x,y)    -w_2(x,y)  \atop  -h_1(x,y) \phantom{-} w_1(x,y)   \right] \cdot \left[u \atop v \right]={ 1 \over D(x,y)}  \left[\tilde{u} \atop \tilde{v} \right],
\end{equation}
where $D(x,y)= w_1(x,y) h_2(x,y) - w_2(x,y) h_1(x,y)$, and
\begin{eqnarray}
\nonumber |\tilde{u}| &<& C \eps (|u|+|v|) < C \eps (\varrho+1) |v|, \\
\nonumber |\tilde{v}| &>& (\kappa - C \eps) |v|-(|g_1(x_1)|+C \eps) |u|>(\kappa - C(1+\varrho) \eps -\varrho K) |v|,
\end{eqnarray}
and $|\tilde{u}| < \varrho |\tilde{v}|$ if $\varrho(\kappa - \varrho K)>C  \eps (\varrho+1)^2$. Furthermore, $|D(x,y)| < 2 (K+C \eps) C \eps$ for some $C>0$ and all $(x,y) \in \hat \Delta$. The lower bound on the operator norm $\| D F^{-1}\|$ on the vertical cone field follows.
\end{proof}

\noindent
As before, for $H_{l_*,\mu}\in W^s(\zeta_\lambda)$, we let $\Omega_n$, $\Gamma_n$ be the domains of the pair 
$${\mathfrak Z_n}=({\mathfrak A}_n,{\mathfrak B}_n)\equiv p\cR^n H_{l_*,\mu}.$$
For brevity, let us also write 
$$\Delta_n\equiv \Omega_n\cup\Gamma_n\text{ and }\Delta'_n\equiv {\mathfrak Z}_n(\Delta_n)\equiv {\mathfrak A}_n(\Omega_n)\cup{\mathfrak B}_n(\Gamma_n).$$

Let $\alpha_*$ be the scaling factor $\alpha_n$ (see definition (\ref{alpha_n})) for the pair $\zeta_\lambda$.
\begin{prop}
\label{prop:hyp1} There exists $N\in\NN$ such that for any  $n\geq N$ we can select  $\delta_0>0$, $k \in \NN$ and $\varrho>0$ so that the following holds. Let $|\mu|<\delta<\delta_0$ and $H_{l_*,\mu}\in W^s(\zeta_\lambda)$. Then 
the derivatives of the inverse branches of the restriction of the pair ${\mathfrak Z}_n$ to the domains $\Delta_n\setminus \Delta_{n+k}$ preserve  the vertical cone field $C^{\text{vert},\varrho}$ and expand vectors in $C^{\text{vert},\varrho}$ at a rate $O(|\alpha_*|^{k} \delta^{-2})$.
\end{prop}
\begin{proof}
Let ${\mathfrak Z}_n=(f_n(x)+\tau_n(x,y),g_n(x)+\chi_n(x,y))$. By  \lemref{lem-preren}, the uniform norms of $\tau_n$ and $\chi_n$ on $\Delta_n$ are bounded from above by $O(\delta^{2})$.

Notice that $\Delta_{n+k}$ is an image of $\Delta_n$ under a linear map which converges to $(\alpha_*^k,0)$ as $n \rightarrow \infty$. Therefore, if $(x,0) \in \left( \Delta_n\setminus \Delta_{n+k} \right) \cap \{y=0\}$, then  
$$ C_2 |\alpha_*|^{n}  >  |x| >C_1 |\alpha_*|^{n+k},$$
for some $C_1$ and $C_2$, which gives
$$ C_4 |\alpha_*|^{n} > |f_n'(x)|, |g_n'(x)| > C_3 |\alpha_*|^{n+k},$$ 
for some $C_3$ and $C_4$. The result follows from \propref{prop:cone-inv} with $\eps=O(\delta^{2})$,  $\kappa=O(|\alpha_*|^{n+k})$ and $K=O(|\alpha_*|^{n})$.
\end{proof}

The following result will be used in the proof of Proposition \ref{prop:gauss}. 
\begin{lem}(L\"owner \cite{Low}) \label{Lowner}
Let $f:\DD \mapsto \DD$ be holomorphic with $f(0)=0$. If $f$ extends to a homeomorphism of $\partial \DD$, then $f$ is a rotation. 
\end{lem}

\noindent
We can now complete the proof of Theorem \ref{thm:C}:
\begin{prop}\label{prop:gauss}
\label{prop:proof}
There exists $\delta>0$ such that the following holds. Let $H_{l_*,\mu}\in W^s(\zeta_\lambda)$ with $|\mu|<\delta$ and let $\gamma_*$ be the invariant curve constructed in \propref{prop-curve}. Then $\gamma_*$ bounds a Siegel disk for $H_{l_*,\mu}$. The eigenvalue $l_*$ is equal to $\lambda_1$,
\begin{equation}
\label{eq:gauss}
\lambda_1=e^{2\pi i\theta_1}\text{ with }\theta=G^m(\theta_1)\text{ for some }m\geq 0.
\end{equation}  
Finally, there exists $\eps_1>0$ such that for all $|\mu|<\eps_1$ and for all $\lambda_1$ satisfying (\ref{eq:gauss}), we have  $H_{\lambda_1,\mu}\in W^s(\zeta_\lambda)$.
\end{prop}
\begin{proof}
Let us select $k$, $N$, and $\varrho$ as in \propref{prop:hyp1}. Let ${i} \geq N$. Fix an open subdomain $\hat \Delta_{i}\Subset \Delta_{i}\cap\Delta_{i}'$. 
Since $H_{l_*,\mu}$ is an $\delta$-small perturbation of the Siegel quadratic polynomial $P_{l_*}$, we can select $\delta>0$ small enough so that the map $H_{l_*,\mu}$ is normally hyperbolic in a sufficiently large neighborhood of the $\alpha$-fixed point of $P_{l_*}$. In particular, by \propref{prop:hyp1}, it is normally hyperbolic in the set $\hat \Delta_{i} \setminus \Delta_{{i}+k}$.  Let $\mathbf q$ be the fixed point of $H_{l_*,\mu}$ which is closest to the $\alpha$-fixed point of $P_{l_*}$. By the Graph Transform, the map $H_{l_*,\mu}$ has a weak stable/unstable/center manifold $W$ of $\mathbf q$ which is $\delta$-close to the slice $\{y=0\}$ (see \cite{HPS}), and therefore  $W \cap \hat \Delta_{i}\neq \emptyset$ if $\delta$ is sufficiently small. 

Let us begin with the case when $\mathbf q$ is attracting. By \propref{prop:hyp1} the inverse branches of ${\mathfrak Z}_{{i}+mk}$, $m \geq 0$ are normally hyperbolic in $\Delta_{{i}+m k} \setminus \Delta_{{i}+(m+1) k}$. Therefore,  the weak attracting submanifold $W$ interesects $\Delta_{{i}+m k}$ for all $m \in \NN$.   We conclude that the invariant curve $\gamma_*$ lies in the closure of $W$.
Applying L\"owners \lemref{Lowner}, we arrive to a contradiction.

Suppose $\mathbf q$ is hyperbolic. Then $W=W^u({\mathbf q})$, the unstable manifold of $\mathbf q$, and successive applications of \propref{prop:hyp1} as above imply that $W$ extends to the invariant curve $\gamma_*$, which is then its boundary. This, again, contradicts  L\"owners \lemref{Lowner}.

Finally, suppose that $\mathbf q$ is semi-neutral (that is, the linear part of the H\'enon map at $\mathbf q$ has a neutral eigenvalue of absolute value 1 and a dissipative eigenvalue of absolute value smaller than 1). In this case $W=W^c({\mathbf q})$: it is only smooth, and {\it a priori}, not uniquely defined. The restriction $H_{l_*,\mu} \arrowvert_W$ is not necessarily holomorphic. 

By density of the irrationals of bounded type in the circle, we can choose a sequence  $H_{l_j,\mu}$ of maps whose neutral eigenvalue $l_j=e^{2 \pi i \vartheta_j}$ for some angle $\vartheta_j \in \RR\setminus \QQ$ of bounded type, converging to $H_{l_*,\mu}$. By continuity of the renormalization operator, for every $M \in \NN$, there exists $J=J(M)$, such that for all $j>J(M)$ $H_{l_j,\mu}$ is ${i}+M k$ times renormalizable  with the height of the renormalizations coinciding with those for the map $H_{l_*,\mu} $. The Siegel  disk $W_j$ of $H_{l_j,\mu}$ is an analytic submanifold of $\C^{2}$.  Applying \propref{prop:hyp1} to the inverse branches of ${\mathfrak Z}_{{i}+mk}^j$, $0 \leq m \leq M$ of $H_{l_j,\mu}$, and using considerations of dominated splitting, we can extend $W_j$ for large $j$ to intersect each $\hat \Delta_{{i}+km}$,  $0 \leq m \leq M$. The rotation numbers of the orbits of points in $W_j \cap \hat \Delta_{{i}+km}$, whose continued fraction expansion is given by the renormalization heights, approach $\theta_1$.  Since, the rotation number  of the orbits of  $H_{l_j,\mu} \arrowvert_{W_j}$ is constant, $\vartheta_j \mapsto \theta_1$, $D H_{l_*,\mu}(\mathbf q)=\lim_{j  \rightarrow \infty} D H_{l_j,\mu}(\mathbf q_j)$, and $l_*=e^{2 \pi i \theta_1}$.  Therefore, $W$ is a Siegel disk  for $H_{l_*,\mu}$, and $H_{l_*,\mu} \arrowvert_W$ is holomorphic. By \propref{prop:hyp1} the submanifold $W$ interesects $\Delta_{{i}+m k}$ for all $m \in \NN$, and, therefore $\gamma_*$ lies in the closure of $W$.  By \propref{prop-curve}, the restriction $H_{l_*,\mu} \arrowvert_{\gamma_*}$ is a homeomorphism, but not a diffeomorphism, therefore $\gamma_*$ cannot lie in  $W$.

Conversely, let $\lambda_1=e^{2\pi i\theta_1}$ satisfy (\ref{eq:gauss}). As shown in Theorem~\ref{th:intersect}, if $\mu$ is small enough, then the family $l\mapsto H_{l,\mu}$ intersects the stable set of  $\zeta_\lambda$ near $P_{\lambda_1}$.
 Denote $l=\lambda_2$ the parameter of the intersection. As we have shown above, if $|\mu|<\eps$, then
$\lambda_2=e^{2\pi i\theta_2}$, where $\theta=G^j(\theta_2)$. The digits in the continued fraction expansion of $\theta_2$ correspond to the periods of renormalizations of $H_{\lambda_2,\mu}$. By considerations of continuity, if $\mu$ is small enough, then the digits in the continued fractions of $\theta_2$ and $\theta_1$ coincide, and hence, $\lambda_2=\lambda_1$.
\end{proof}


\end{document}